\documentclass[english, 10pt]{amsart}

\usepackage{babel}
\usepackage{amstext}
\usepackage{amsmath}
\usepackage{amssymb}
\usepackage{amsfonts}
\usepackage{latexsym} 
\usepackage{ifthen}
\usepackage{xypic}
\xyoption{all}
\usepackage{verbatim}

\newcommand{\Bs}{{\rm Bs}}

\newtheorem{lemma1}{}[section]

\newenvironment{lemma}{\begin{lemma1}{\bf Lemma.}}{\end{lemma1}}
\newenvironment{example}{\begin{lemma1}{\bf Example.}\rm}{\end{lemma1}}

\newenvironment{theorem}{\begin{lemma1}{\bf Theorem.}}{\end{lemma1}}
\newenvironment{proposition}{\begin{lemma1}{\bf Proposition.}}{\end{lemma1}}
\newenvironment{corollary}{\begin{lemma1}{\bf Corollary.}}{\end{lemma1}}

\newenvironment{remarks}{\begin{lemma1}{\bf Remarks.}\rm}{\end{lemma1}}
\newenvironment{definition}{\begin{lemma1}{\bf Definition.}}{\end{lemma1}}

\newenvironment{remark*}{{\bf Remark.}}{}
\newenvironment{example*}{{\bf Example.}}{}

\newcommand{\R}{\ensuremath{\mathbb{R}}}
\newcommand{\Q}{\ensuremath{\mathbb{Q}}}

\newcommand{\C}{\ensuremath{\mathbb{C}}}
\newcommand{\N}{\ensuremath{\mathbb{N}}}
\newcommand{\PP}{\ensuremath{\mathbb{P}}}

\newcommand{\holom}[3]{\ensuremath{#1\colon #2  \rightarrow #3}}
\newcommand{\fibre}[2]{\ensuremath{#1^{-1} (#2)}}

\makeatletter
\ifnum\@ptsize=0  \fi \ifnum\@ptsize=2  \fi 

\newcommand\sI{{\mathcal I}}

\newcommand\sO{{\mathcal O}}

\DeclareMathOperator*{\Nklt}{Nklt}

\newcommand\pic{\ensuremath{\mbox{Pic}}}

\newcommand{\Chow}[1]{\ensuremath{\mathcal{C}(#1)}}

\newcommand{\Hilb}{\ensuremath{\mathcal{H}}}

\usepackage{color}
\definecolor{red-}{rgb}{1.0,0.0,0.0}
\definecolor{grey}{rgb}{0.6, 0.6, 0.6}
\definecolor{brown}{rgb}{0.5,0.2,0.0}
\definecolor{blue-}{rgb}{0.0,0.1,1.0}
\definecolor{green-}{rgb}{0.0, 0.6, 0.0}
\definecolor{gold}{rgb}{0.8,0.7,0.0}
\definecolor{black}{rgb}{0.0,0.0,0.0}
\definecolor{DarkGreen}{rgb}{0.0,0.3,0.2}
\definecolor{LightGreen}{rgb}{0.8,1.0, 0.8}
\definecolor{yellow}{rgb}{0.9,0.9,0.0}

\title{Fano varieties with small non-klt locus} 
\date{September 20, 2013}

\author{Mauro C. Beltrametti}
\author{Andreas H\"oring}
\author{Carla Novelli}


\address{Mauro C. Beltrametti, Dipartimento di Matematica, Universit\`a degli Studi di Genova, Via Dodecaneso 35, I-16146 Genova, Italy}
\email{beltrame@dima.unige.it}

\address{Andreas H\"oring, Laboratoire de Math{\'e}matiques J.A. Dieudonn{\'e},
UMR 7351 CNRS, Universit{\'e} de Nice Sophia-Antipolis, 06108 Nice Cedex 02, France        
}
\email{hoering@unice.fr}

\address{Carla Novelli, Dipartimento di Matematica, Universit\`a degli Studi di Padova, via Trie\-ste 63, I-35121 Padova, Italy}
\email{novelli@math.unipd.it}

\begin{document}

\begin{abstract}
Let $X$ be a Fano variety of index $k$.  
Suppose that the non-klt locus $\Nklt(X)$ is not empty. We prove that
$\dim \Nklt(X) \geq k-1$ and equality holds if and only if 
$\Nklt(X)$ is a linear $\PP^{k-1}$.
In this case $X$ has  lc singularities and is a generalised cone
with $\Nklt(X)$ as vertex.

If $X$ has lc singularities and $\dim \Nklt(X)=k$ we describe the non-klt locus $\Nklt(X)$ and the global geometry of $X$. 
Moreover, we construct examples to show that all the classification results are effective.
\end{abstract}

\maketitle

\section{Introduction}

Let $X$ be a Fano variety, i.e. $X$ is a normal projective complex variety such that the anticanonical 
divisor $-K_X$ is Cartier and ample. 
While Fano varieties with mild singularities (terminal or canonical)
have been studied by many authors, the goal of this paper is to study Fano varieties whose
non-klt locus $\Nklt(X)$ is not empty. 
 A cone over a smooth variety $Y$ with trivial anticanonical
divisor provides a simple example of such a variety. 
Using methods from the minimal model program, Ishii \cite{Ish91, Ish94} characterised
cones as the only Fano varieties having a {\em finite} non-klt locus:

\begin{theorem} \cite{Ish91} \label{theoremishii}
Let $X$ be a Fano variety of dimension $n$ such that $\Nklt(X)$ is not empty and finite. 
Then $X$ is a cone over a variety $Y$ of dimension $n-1$ such that $Y$ has canonical singularities and $K_Y$ is trivial.
\end{theorem}

Recall now that the index of a Fano variety $X$ is defined as
$$
\sup \{ 
m \in \N \ | \ \exists \ H \in \pic(X) \ \mbox{s.t.} -K_X \simeq m H
\}.
$$
Making stronger assumptions on the singularities of $X$, the first-named author and Sommese 
established a basic relation between the index and the dimension of the non-klt locus:

\begin{theorem} \cite{BS87} \label{theoremBS}
Let $X$ be a Fano variety of index $k$ that is Cohen--Macaulay and
such that $-K_X \simeq k H$, with $H$ a very ample Cartier divisor on~$X$. 
Suppose that the non-klt locus  $\Nklt(X)$ is not empty. Then we have
$$
\dim \Nklt(X) \geq k-1,
$$
and equality holds if and only if $\Nklt(X)$ is a linear $\PP^{k-1}$ and $X$ is a generalised cone with $\Nklt(X)$ as  vertex.
\end{theorem}

Unfortunately the Cohen--Macaulay condition excludes an important number of examples: cones
over Calabi--Yau manifolds are Cohen--Macaulay, but a cone over an abelian variety is not.
Our first result generalises Theorem \ref{theoremishii}
and Theorem \ref{theoremBS}:

\begin{theorem} \label{theorembasicestimate}
Let $X$ be a Fano variety of index $k$ and such that $-K_X \simeq k H$, with $H$  a Cartier divisor on~$X$.
Suppose that the non-klt locus $\Nklt(X)$ is not empty. Then we have
$$\dim \Nklt(X) \geq k-1,$$ and equality holds if and only if 
$
(\Nklt(X),\sO_{\Nklt(X)}(H)) \cong (\PP^{k-1},\sO_{\PP^{k-1}}(1))
$.
In this case $X$ has  lc singularities and is a generalised cone
with $\Nklt(X)$ as  vertex.
\end{theorem}

A remarkable feature that is common to the Theorems \ref{theoremishii}, \ref{theoremBS} and \ref{theorembasicestimate}
is that the property $\dim \Nklt(X) = k-1$ implies that the singularities of $X$ are  log-canonical.
If $\dim \Nklt(X) = k$ this is no longer the case (cf. Example \ref{examplenonlc}),
so for the rest of the paper we will make the additional assumption that $X$ is a Fano variety
with  lc singularities. 
Since we assume $K_X$ to be Cartier the non-klt locus
coincides with the locus of irrational singularities \cite[Cor.5.24]{KM98}, so we can study $\Nklt(X)$  both in terms of
discrepancies and using cohomological methods.
Note also that Theorem \ref{theorembasicestimate} naturally separates into 
a local part, i.e. the description of the non-klt locus $\Nklt(X)$, and a global part, i.e. the description
of the geometry of $X$.  
For the local part we use a subadjunction argument to describe the low-dimensional
lc centres:

\begin{theorem} \label{theoremnkltlocus}
Let $X$ be a Fano variety of index $k$ with  lc singularities and such that  $-K_X \simeq k H$,  with $H$   a Cartier divisor on~$X$.
If $\dim \Nklt(X) = k$ we have 
$$
(\Nklt(X),\sO_{\Nklt(X)}(H)) \cong (\PP^{k},\sO_{\PP^{k}}(1))
$$
or
 $$
(\Nklt(X),\sO_{\Nklt(X)}(H)) \cong (Q^{k},\sO_{Q^{k}}(1)),
$$
where $Q^k \subset \PP^{k+1}$ is a (possibly reducible) hyperquadric. 
\end{theorem}

This local result gives some interesting global information:

\begin{corollary}
Let $X$ be a Fano variety of index $k$ with  lc singularities and such that $\dim \Nklt(X) = k$. Then $X$ is rationally chain-connected.
\end{corollary}

Indeed by a result of Broustet and Pacienza \cite[Thm.1.2]{BP11} the variety $X$ is rationally chain-connected modulo
the non-klt locus. Since $\Nklt(X)$ is rationally chain-connected the statement follows.

We continue the study of the case $\dim \Nklt(X)=k$ by considering a terminal modification $X' \rightarrow
X$ (cf. Definition \ref{definitionterminal}). 
If $\Nklt(X)$ is a linear $\PP^{1}$ the variety $X'$ can be rationally connected and have a very rich birational
geometry, cf. \cite{Ish91, Ish94}. If $\Nklt(X)$ is a conic the variety $X'$ is never rationally connected and
we obtain a precise classification:

\begin{theorem} \label{theoremkone}
Let $X$ be a Fano variety (of index $1$) and dimension $n \geq 3$ with lc singularities. Suppose that $\Nklt(X)$ is a curve
and $-K_X \cdot \Nklt(X)=2$. Let $\mu: X' \rightarrow X$ be a terminal modification. 
Then the base of the MRC-fibration $X' \dashrightarrow Z$ has dimension $n-2$ and the general fibre
$F$ polarised by $\sO_F(- \mu^* K_X)$ is a linear $\PP^2$, a quadric, a Veronese surface or a ruled surface.
\end{theorem}

In fact we know much more. The base $Z$ has an effective canonical divisor, moreover we can construct
birational models of $X$ which have a very simple fibre space structure:

\begin{proposition} \label{propositionfibrespace}
Let $X$ be a Fano variety of dimension $n$ and index $k$ with  lc singularities such that $\dim \Nklt(X) = k$. Suppose that the base of the MRC-fibration $X' \dashrightarrow Z$ has dimension $n-k-1$. Then there exists 
a normal projective variety $\Gamma$ admitting a birational morphism
\holom{p}{\Gamma}{X} and an equidimensional fibration \holom{q}{\Gamma}{\Hilb} onto a projective  manifold ${\mathcal H}$ such that one of the following holds:
\begin{enumerate}
\item the fibration $q$ is a projective bundle;
\item the fibration $q$ is a quadric fibration;
\item the general $q$-fibre is a Veronese surface.
\end{enumerate}
\end{proposition}

In Section \ref{sectionexamples} we show that all the classification results in this paper 
are effective, i.e. there exist examples realising all the cases in Theorem \ref{theoremkone} and Proposition \ref{propositionfibrespace}.
The condition $\dim Z=n-k-1$ may seem rather ad-hoc, but we prove that it is satisfied if $\Nklt(X)$ is a
quadric and $k=1$ or $k \geq n-3$ (cf. Proposition \ref{propositionladder}).
Actually we prove that if $X$ admits a ladder (in the sense of Fujita, cf. Definition \ref{definitionladder}), then $H^{n-k-1}(X', \sO_{X'}) \neq 0$. We expect that any Fano variety with lc singularities admits a ladder, but this depends on the
difficult non-vanishing conjecture \cite[Conj.2.1]{Kaw00}.

{\bf Acknowledgements.}  The first and the third authors thank the Research Network Program ``GDRE-GRIFGA'' for partial support.
The second-named author was partially supported by the A.N.R. project ``CLASS''. 

\begin{center}
{\bf Notation}
\end{center}

We work over the complex numbers, topological notions always refer to the Zariski topology.
For general definitions we refer to \cite{Har77}.
The dimension of an algebraic variety is defined as the maximum of the dimension of its irreducible components.

On a normal variety we will denote by $\simeq$ the linear equivalence of {\em Cartier} divisors,
while $\sim_\Q$ (resp. $\equiv$) will be used for the $\Q$-linear (resp. numerical) equivalence
of $\Q$-Cartier $\Q$-divisors.

We will frequently use the terminology and results 
of the minimal model program (MMP) as explained in \cite{KM98, Deb01, Kol13}. In particular klt stands for ``Kawamata log terminal", 
dlt for ``divisorial log terminal", plt for ``purely log terminal", and lc for ``log canonical'' singularities.

\begin{definition} \label{definitioncanonical}
Let $X$ be a normal variety. The {\rm canonical modification} of $X$
is the unique projective birational morphism \holom{\mu}{X'}{X} from a normal
variety $X'$ with  canonical singularities such that $K_{X'}$ is $\mu$-ample.
\end{definition}

\begin{remarks} \label{remarkscanonical}
The existence of the canonical modification is a consequence
of \cite{BCHM10}, cf. the forthcoming book \cite{Kol13}. If $K_X$ is $\Q$-Cartier we have
$$
K_{X'} \sim_\Q \mu^* K_X - E,
$$
where $E$ is an effective $\Q$-divisor whose support is the exceptional locus of $\mu$.

Let us recall that a normal variety is Gorenstein if it is Cohen--Macaulay and the canonical divisor is Cartier.
Since canonical singularities are Cohen--Macaulay the first condition is empty for $X'$, so the Gorenstein
locus of $X'$ is the open subset where $K_{X'}$ is Cartier.
\end{remarks}

An important advantage of canonical models is their uniqueness, but their singularities can be rather complicated. 
One can improve the singularities by losing uniqueness:
  
\begin{definition} \label{definitionterminal}\cite{Kol13}
Let $X$ be a normal variety. A {\rm terminal modification} of $X$
is a projective birational morphism \holom{\mu}{X'}{X} from a normal
variety $X'$ with terminal singularities such that $K_{X'}$ is $\mu$-nef.
\end{definition}

We will use the standard definitions and results of adjunction theory from \cite{Fuj90, BS95},
except the following generalised version of the nefvalue of a Mori contraction:

\begin{definition} \label{definitionnefvalue}
Let $X$ be a normal quasi-projective variety with lc singularities, and
let $H$ be a nef and big Cartier divisor on $X$.
Let $\holom{\varphi}{X}{Y}$ be the contraction of a $K_X$-negative extremal 
ray $R$ such that $H \cdot R>0$. The {\rm nefvalue} $r := r(\varphi, H)$
is the positive number such that
$(K_X+r H) \cdot R=0$.
\end{definition}

Let us recall a well-known consequence of the cone theorem:

\begin{lemma} \label{lemmacontraction}
 Let $X$ be a normal projective variety
with klt singularities such that
$$
-K_X \sim_\Q N + E,
$$
where $N$ is a nef $\Q$-Cartier $\Q$-divisor and $E$ is a non-zero effective $\Q$-divisor. Suppose also that for every curve $C \subset X$ such that $N \cdot C=0$ we have $K_X \cdot C \geq 0$. 
Then there exists a $K_X$-negative extremal ray $R$ such that $E \cdot R>0$ and $N \cdot R>0$.
\end{lemma}

\begin{proof} 
Arguing by contradiction we suppose that $-E \cdot R \geq 0$ for all $K_X$-negative extremal rays. If $C \subset X$
is a curve such that $K_X \cdot C \geq 0$, then 
$$
-E \cdot C = K_X \cdot C + N \cdot C \geq 0.
$$
By the cone theorem
this implies that the antieffective divisor $-E$ is nef, a contradiction. The property $N \cdot R>0$ is trivial since
$N$ is positive on all $K_X$-negative curves. 
\end{proof}

\section{Proof of Theorem \ref{theorembasicestimate}}

Before we can prove Theorem \ref{theorembasicestimate} we need a technical lemma which 
replaces an inaccurate statement\footnote{The statement is $\dim F \geq r$, but the proof only yields $\dim F \geq \lfloor r \rfloor$.} in \cite[Thm.2.1(II,i)]{And95}.

\begin{lemma} \label{lemmarounddown}
Let $X$ be a normal quasi-projective variety with  canonical singularities,
and let $H$ be a nef and big Cartier divisor on $X$.
Let $\holom{\varphi}{X}{Y}$ be a birational projective morphism with connected fibres
onto a normal variety $Y$ such that $H$ is $\varphi$-ample and 
$K_X+rH$ is $\varphi$-numerically trivial for some $r>0$.
Fix a point $y \in Y$ and suppose that all the irreducible components
of $\fibre{\varphi}{y}$ have dimension at most $\lfloor r \rfloor$.
Suppose that there exists an irreducible component $F \subset \fibre{\varphi}{y}$
that meets the Gorenstein locus of $X$.
Then we have $\lfloor r \rfloor=r$.
\end{lemma}

\begin{proof}
The statement is local on the base, so we can assume that $Y$ is affine. In particular
every $\varphi$-generated line bundle is globally generated.

By \cite[Thm.,p.740]{AW93} we know that $H$ is $\varphi$-globally generated, moreover by
\cite[Thm.2.1(II,ii)]{And95} all the irreducible components of $\fibre{\varphi}{y}$
are isomorphic to $\PP^{\lfloor r \rfloor}$ and $H|_F$ is a hyperplane divisor.

We will proceed by induction on $\lfloor r \rfloor$.
For the case $\lfloor r \rfloor=1$ note that by \cite[Lemma 1]{Ish91}
we have $K_X \cdot F \geq -1$. Since $H \cdot F=1$ we obtain that $r \leq 1$. 
For the induction step choose $X' \in |H|$ a general divisor,
and consider the induced birational morphism $\holom{\varphi|_{X'}}{X'}{\varphi(X')}$.
Since $H$ is $\varphi$-globally generated the variety $X'$ is normal with
 canonical singularities, moreover the fibre \fibre{\varphi|_{X'}}{y}
has pure dimension $\lfloor r \rfloor - 1$ and $F \cap X'$ meets the Gorenstein
locus of $X'$. By adjunction we know that $K_{X'}+(r-1) H|_{X'}$ is $\varphi|_{X'}$-numerically trivial,
so by the induction hypothesis we have $\lfloor r - 1 \rfloor = r - 1$.
\end{proof}

\begin{proof}[Proof of Theorem \ref{theorembasicestimate}]
Let $\holom{\mu}{X'}{X}$ be the canonical modification of $X$ (cf. Definition \ref{definitioncanonical}). Then we have
$$
K_{X'} = \mu^* K_X + \sum_{E_i \ \mbox{\tiny $\mu$-exc.}} a_i E_i,
$$
and $a_i \leq -1$ for all $i$. 
Set $E:=- \sum_{E_i \ \mbox{\tiny $\mu$-exc.}} a_i E_i$, then $E$ is an effective divisor mapping onto $\Nklt(X)$. Since $K_X$ is Cartier, the non-Gorenstein locus of $X'$ is contained in $E$.

By Lemma \ref{lemmacontraction} there exists a $K_{X'}$-negative extremal ray $R$ on $X'$ such that
$E \cdot R>0$ and $\mu^* H \cdot R>0$. Let $\holom{\varphi}{X'}{Y}$ be the corresponding Mori contraction, and 
let $F$ be an irreducible component of a positive-dimensional $\varphi$-fibre.
Since $E$ is $\varphi$-ample, the intersection $E \cap F$ is non-trivial.
Since $E$ is $\Q$-Cartier we have
$$
\dim (F \cap E) \geq \dim F - 1,
$$
and equality holds if and only if $F \not\subset E$.
Since $K_X$ is $\mu$-ample and $\varphi$-antiample the morphism
$$
(F \cap E) \rightarrow \mu(F \cap E) 
$$
is finite, thus we have
\begin{equation} \label{eqndimfibres}
\dim \Nklt(X) = \dim \mu(E) \geq \dim \mu(F \cap E) = \dim (F \cap E) \geq \dim F -1, 
\end{equation}
and equality holds if and only if $\dim F= \dim \Nklt(X) +1$ and $F \not\subset E$.

Since $-K_{X'} \sim_\Q k \mu^* H + E$ we see that the nefvalue $r:=r(\varphi, H)$ (cf. Definition \ref{definitionnefvalue})
is strictly larger than $k$. By \cite[Thm.2.1(I,i)]{And95} this implies that $\dim F \geq r-1>k-1$.
By \eqref{eqndimfibres} we get $\dim \Nklt(X) > k-2$. Since $\dim \Nklt(X)$ is an integer the inequality in the statement
follows.

Suppose now that we are in the boundary case $\dim \Nklt(X)=k-1$. Arguing by contradiction 
we suppose that the extremal contraction $\varphi$ is birational. 
By \cite[Thm.2.1(II,i)]{And95} we then have $\dim F \geq \lfloor r \rfloor \geq k$. 
Thus equality holds in \eqref{eqndimfibres} and the variety $F$ is not contained in $E$. 
Therefore we satisfy
the conditions of Lemma \ref{lemmarounddown} and get $\lfloor r \rfloor = r$, so $r=k$ by (\ref{eqndimfibres}).
However by construction we have $r>k$, a contradiction.

Thus the contraction $\varphi$ is of fibre type; moreover all the irreducible components of each $\varphi$-fibre have dimension exactly $k$. By \cite[Thm.2.1(I,ii)]{And95} this implies that
$X' \rightarrow Y$ is a $\PP^k$-bundle. Let now $F \cong \PP^k$ be a general $\varphi$-fibre. Then $F$ is contained
in the smooth locus of $X'$, in particular all the divisors $E_i$ are Cartier in a neighborhood of $F$. By adjunction we see that
$$
\sO_F\big(\sum a_i E_i\big) \simeq \sO_{\PP^k}(-1).
$$ 
Since the left hand side is a sum of antieffective Cartier divisors, we see that (up to renumbering) we have $a_1=-1$ and $E_1 \cap F$ is a hyperplane. Moreover we have $E_i \cap F= \emptyset$ for all $i \geq 2$. In particular for every $i \geq 2$ we have
$E_i \simeq \varphi^* D_i$ with $D_i$ a Weil divisor on $Y$. However if we take $y \in D_i$ a point,
then $\fibre{\varphi}{y} \subset E_i$ and $-K_{X'}|_F$ is ample, so $F \rightarrow \mu(F) \subset \mu(E_i)$ is finite.
Since we assumed that $\dim \Nklt(X)= \dim \mu(E) =k-1$ this yields $E_i=0$ for all $i \geq 2$. Thus we have
$$
K_{X'} = \mu^* K_X - E_1.
$$
One easily sees that the pair $(X', E_1)$ is lc, so $X$ has lc singularities. The generalised cone structure is given
by the maps $\mu$ and $\varphi$.
\end{proof}

\begin{example} \label{examplenonlc}
Let $Y \subset \PP^3$ be a cone over a smooth quartic curve in $\PP^2$, then
$K_Y$ is trivial and the vertex is the unique irrational point on $Y$. Clearly
$Y$ does not have lc singularities, so if $X$ is the cone over $Y$,
then $X$ is a Fano variety of index one such that $\Nklt(X)$ has dimension one
and $X$ is not lc.
\end{example}

\section{Description of the non-klt locus}

Let $X$ be a normal quasi-projective variety, and let $\Delta$ be an effective boundary divisor on $X$
such that $K_X+\Delta$ is $\Q$-Cartier and the pair $(X, \Delta)$ is lc. Let $\Nklt(X, \Delta)$ be the non-klt locus. 
Recall that a subvariety $W \subset X$ is an lc centre
of the pair $(X, \Delta)$ if there exists a birational morphism $\holom{\mu}{X'}{X}$
and an effective divisor $E \subset X'$ of discrepancy $-1$ such that $\mu(E)=W$.
Recall also that the intersection $W_1 \cap W_2$ of two lc centres $W_i$ is a union
of lc centres \cite[Prop.1.5]{Kaw97}. In particular, given a point $x \in \Nklt(X, \Delta)$, there
exists a unique lc centre $W$ passing through $x$ that is minimal with respect to the inclusion.
By \cite[Thm.1.6]{Kaw97} the lc centre $W$ is normal in the point $x$. If $W$ is minimal (in every 
point $x \in W$), the Kawamata subadjunction formula holds \cite{Kaw98, FG12}:
there exists an effective $\Q$-divisor $\Delta_W$ on $W$ such that the pair $(W, \Delta_W)$
is klt and
\begin{equation} \label{subadjunction}
K_{W}+\Delta_{W} \sim_\Q  (K_X+\Delta)|_W.
\end{equation}
The following weak form of the subadjunction formula for non-minimal lc centres 
is well-known to experts, we nevertheless include a proof for lack of reference:

\begin{lemma} \label{lemmasubadjunction}
Let $(X, \Delta)$ be a projective lc pair, and let $W \subset X$ be an lc centre. 
Let $\holom{\nu}{W^n}{W}$ be the normalisation. Then there exists an effective divisor $\Delta_{W^n}$ on $W^n$
such that
$$
K_{W^n}+\Delta_{W^n} \sim_\Q  \nu^* (K_X+\Delta)|_W.
$$
Suppose that $Z \subset W$ is an lc centre such that $\dim Z=\dim W-1$.
Then we have a set-theoretical inclusion
$$
\fibre{\nu}{Z} \subset \Delta_{W^n}.
$$
\end{lemma}

\begin{proof}
Let $\holom{\mu}{(X^m, \Delta^m)}{(X, \Delta)}$ be a dlt-model, i.e. $\mu$ is birational, the pair $(X^m, \Delta^m)$ is dlt and
$K_{X^m}+\Delta^m \sim_\Q \mu^*(K_X+\Delta)$ \cite[Thm.3.1]{KK10}. Let $S \subset X^m$  
be an lc centre of $(X^m, \Delta^m)$ that dominates $W$ and that is minimal with respect to the inclusion. 
By \cite[Thm.1]{Kol11} there exists an effective divisor $\Delta_S$ such that $(S, \Delta_S)$ is dlt and
$K_S+\Delta_S \sim_\Q \mu_S^* (K_X+\Delta)|_W$, where $\holom{\mu_S}{S}{W}$ is the restriction
of $\mu$ to $S$. Moreover $(S, \Delta_S)$ is klt on the generic fibre of $\mu_S$.
The variety $S$ being normal, the morphism $\mu_S$ factors through the normalisation $\nu$, and we denote by
$\holom{\mu_S^n}{S}{W'}$ and $\holom{\tau}{W'}{W^n}$ the Stein factorisation.

Moreover,
\begin{equation}\label{firstformula}
K_S+\Delta_S \sim_\Q (\mu^n_S)^* \circ \tau^* \circ \nu^* (K_X+\Delta)|_W
\end{equation}
implies that $\mu_S^n$ is an lc-trivial fibration in the sense of \cite[Defn.2.1]{Amb04}
and we denote by $\Delta_{W'}$ the discriminant divisor. Up to replacing $\mu_S^n$
by a birationally equivalent fibration we know by inversion of adjunction \cite[Thm.3.1]{Amb04}
that the pair $(W', \Delta_{W'})$ is lc. 

Using the terminology of \cite{Amb05} (see in particular Definition 3.2 and Theorem 3.3)  the moduli b-divisor 
of the lc-trivial fibration is $b$-nef and good, in particular
it has non-negative Kodaira dimension. Thus there exists an effective divisor $E$ such that  
$$
K_S+\Delta_S \sim_\Q (\mu_S^n)^* (K_{W'} + \Delta_{W'} + E).
$$
By the proof of \cite[Lemma 1.1]{FG12} there exists an effective divisor $\Delta_{W^n}$ such that
$$
K_{W'} + \Delta_{W'} + E \sim_\Q \tau^* (K_{W^n}+\Delta_{W^n}).
$$ 
Recalling \eqref{firstformula} we derive
$$
K_{W^n}+\Delta_{W^n} \sim_\Q \nu^* (K_X+\Delta)|_W.
$$

Suppose now that $Z \subset W$ is an lc centre such that $\dim Z=\dim W-1$.
Then we know by \cite[Cor.11]{Kol11} that 
every irreducible component of 
$\fibre{(\nu \circ \tau)}{Z}$ is an lc centre of the pair $(W', \Delta_{W'})$. Since $\fibre{(\nu \circ \tau)}{Z}$ is a divisor
in $W'$ we have a set-theoretical inclusion
\begin{equation} \label{inclusionupstairs}
\fibre{(\nu \circ \tau)}{Z} \subset \Delta_{W'}.
\end{equation}
Since the pair $(W^n, \Delta_{W^n})$ is not klt in the points where the pair $(W', \Delta_{W'}+E)$ is not klt (cf. \cite[Prop.20.3]{Uta92}), 
the inclusion \eqref{inclusionupstairs} implies a set-theoretical inclusion
$\fibre{\nu}{Z} \subset \Delta_{W^n}$.
\end{proof}

For the description of the non-klt locus
we start with a refinement of the local part of Theorem \ref{theorembasicestimate}
in the log-canonical case:

\begin{lemma} \label{lemmaestimate}
Let $(X, \Delta)$ be a projective lc pair such that $-(K_X+\Delta) \sim_\Q k H$, with $H$ an ample Cartier divisor
and $k \geq 1$.
Let $W \subset X$ be an lc centre. Then we have
$$
\dim W \geq k-1,
$$
and equality holds if and only if $\lfloor k \rfloor = k$ and $(W,H|_W) \cong (\PP^{k-1},\sO_{\PP^{k-1}}(1))$.
\end{lemma}

\begin{proof}
The statement is trivial if $k=1$, so suppose $k>1$.
It is sufficient to prove the statement for $W \subset X$ a minimal centre. 
By the subadjunction formula \eqref{subadjunction} there exists a divisor $\Delta_W$ on $W$ such that $(W, \Delta_W)$ is klt
and 
$$
K_W + \Delta_W \sim_\Q (K_X+\Delta)|_W \sim_\Q -k H|_W.
$$
If $\dim W>0$ we can apply \cite[Thm.2.5]{AD12} 
to see that the log Fano variety $(W, \Delta_{W})$ has dimension at least $k-1$ and
equality holds if and only if $(W,H|_W) \cong (\PP^{k-1},\sO_{\PP^{k-1}}(1))$.
Thus we are left to exclude the possibility that $\dim W=0$:
since $-H-(K_X+\Delta) \sim_\Q (k-1) H$ is ample, the restriction map
$$
H^0(X, \mathcal O_X(-H)) \to H^0(W, \mathcal O_W(-H))
$$
is surjective by \cite[Thm.2.2]{Fuj11}. If $W$ is a point, the space $H^0(W, \mathcal O_W(-H))$ is not zero,
which is impossible since the antiample divisor $-H$ has no global sections on $X$.
\end{proof}

The following proposition is the key step in the proof of Theorem \ref{theoremnkltlocus}:

\begin{proposition} \label{propositioncentrek}
Let $(X, \Delta)$ be a projective lc pair of dimension $n \geq 3$ such that $-(K_X+\Delta) \sim_\Q k H$, with $H$ an ample Cartier divisor
and $k \in \N$. Let $W$ be an lc centre of $(X, \Delta)$ of dimension $k$.

If $W$ is minimal, then $(W,H|_W) \cong (\PP^{k},\sO_{\PP^{k}}(1))$ or 
$(W,H|_W) \cong (Q^{k},\sO_Q(1))$, with $Q^k \subset \PP^{k+1}$ an integral quadric.

If $W$ is not minimal, then $(W,H|_W) \cong (\PP^{k},\sO_{\PP^{k}}(1))$
and $W$ contains exactly one lc centre $Z \subsetneq W$.
\end{proposition}

\begin{proof}
If $W$ is minimal we know by the subadjunction formula \eqref{subadjunction}
that there exists an effective divisor $\Delta_{W}$ on $W$
such that  $(W, \Delta_{W})$ is log Fano of index $k$ and dimension $k$.
The statement then follows from \cite[Thm.2.5]{AD12}.

Suppose now that $W$ is not minimal. Then $W$ contains another lc centre $Z \subsetneq W$
and by Lemma \ref{lemmaestimate} we know that $Z$ has dimension $k-1$.
Denote by $\holom{\nu}{W^n}{W}$ the normalisation.
By Lemma \ref{lemmasubadjunction} 
there exists an effective divisor $\Delta_{W^n}$
such that
\begin{equation} \label{help1}
K_{W^n}+\Delta_{W^n} \sim_\Q  \nu^* (K_X+\Delta)|_W \sim_\Q -k \nu^* H|_W.
\end{equation}
Thus the pair $(W^n, \Delta_{W^n})$ is log Fano of dimension $k$ and index $k$,
moreover by the last part of Lemma \ref{lemmasubadjunction} 
we have a set-theoretical inclusion
$$
\fibre{\nu}{Z} \subset \Delta_{W^n}.
$$
In particular $\Delta_{W^n}$ is not empty, so by
\cite[Thm.2.5]{AD12} we obtain
$W^n \cong \PP^{k}$ and $\sO_{W^n}(\nu^* H) \simeq \sO_{\PP^{k}}(1)$.
By \eqref{help1} this implies $\sO_{W^n}(\Delta_{W^n}) \simeq \sO_{\PP^{k}}(1)$,
so we see that $\fibre{\nu}{Z}=\Delta_{W^n}$ and $\fibre{\nu}{Z}$ is a hyperplane. 
Note that this already implies that $Z$ is unique.

Therefore it remains to prove that $W$ is normal. Note
first that $W$ is minimal, hence normal, in the complement of $Z$. Thus it  is sufficient to prove that $W$ is normal 
in every point $x \in Z$. 
By Lemma \ref{lemmaestimate} we have $Z \cong \PP^{k-1}$, so we get 
a finite map
$$
\holom{\nu|_{\fibre{\nu}{Z}}}{\fibre{\nu}{Z} \cong \PP^{k-1}}{Z \cong \PP^{k-1}}.
$$
Since $\sO_{\fibre{\nu}{Z}}(\nu^* H) \simeq  \sO_{\PP^{k-1}}(1)$ this map has degree one, so it is an isomorphism. This proves that the normalisation $\nu$ is an injection on points.
By a result of Ambro \cite[Thm.1.1]{Amb11} we know that $W$ is semi-normal, so $\nu$ is an isomorphism.
\end{proof}

\begin{proof}[Proof of Theorem \ref{theoremnkltlocus}]
By Nadel's vanishing theorem (see e.g. \cite[Thm.3.2]{Fuj11}) we have $H^1(X, \sI_{\Nklt(X)})=0$, so 
the map
$$
H^0(X, \sO_X) \rightarrow H^0(\Nklt(X), \sO_{\Nklt(X)})
$$  
is surjective. In particular the non-klt locus $\Nklt(X)$ is connected.
If $\Nklt(X)$ is irreducible we conclude by Proposition \ref{propositioncentrek}.

Suppose from now on that $\Nklt(X)$ is reducible. If $W_1$ (resp. $W_2$) is an irreducible component
of $\Nklt(X)$, hence an lc centre, of dimension $r_1$ (resp. $r_2$), the intersection
$W_1 \cap W_2$ is either empty or a union of lc centres of dimension at most 
$\min(r_1, r_2)-1$. By the connectedness of $\Nklt(X)$  we can reduce to consider the  second case.  Then Lemma \ref{lemmaestimate}  implies that every irreducible component
of $\Nklt(X)$ has dimension $k$; moreover two components meet along a set of dimension $k-1$.
By Lemma \ref{propositioncentrek} every irreducible component $W_i$ is isomorphic to $\PP^k$ 
and contains exactly one lc centre, so we see that $\Nklt(X)$ has exactly two irreducible components.
These two irreducible components meet along an lc centre of dimension $k-1$,
so by Lemma \ref{lemmaestimate} the intersection $W_1 \cap W_2$ is a linear $\PP^{k-1}$.
Thus $\Nklt(X)=W_1 \cup W_2$ is a reducible quadric of dimension $k$.
\end{proof}

\section{Description of the geometry of the Fano variety}
\label{sectiondescriptionX}

We start by classifying certain weak log Fano varieties that will appear as the fibres 
of the MRC-fibration:

\begin{proposition} \label{propositionindexk}
Let $X$ be a normal projective variety of dimension $k+1 \geq 2$ with  terminal singularities,
and let $\Delta$ be a non-zero effective Weil $\Q$-Cartier divisor on $X$ such that the pair $(X, \Delta)$ is lc,  and
$$
-(K_X+\Delta) \sim_\Q k H
$$
with $H$ a nef and big Cartier divisor on $X$. Suppose also that for every curve $C \subset X$
such that $H \cdot C=0$ we have $K_X \cdot C \geq 0$. Then $(X, \sO_X(H))$ is one of the following quasi-polarised varieties:
\begin{enumerate}
\item $(\mathbb P^{k+1}, \mathcal O_{\mathbb P^{k+1}}(1))$ and $\Delta$ is a quadric; or
\item $(Q^{k+1}, \mathcal O_{Q^{k+1}}(1))$ and $\Delta$ is a quadric; or
\item a generalised cone of dimension $k+1$ over the Veronese surface $(\mathbb P^2, \mathcal O_{\mathbb P^2}(2))$ and $\Delta$
is the generalised cone of dimension $k$ over $(\mathbb P^1, \mathcal O_{\mathbb P^1}(2))$; or
\item a scroll $\PP(\sO_{\PP^1}(a) \oplus \sO_{\PP^1}(1) \oplus \sO_{\PP^1}^{\oplus k-1})$ where $a \in \N_{>0}$
and $\Delta$ is the union of  $\PP(\sO_{\PP^1}(1) \oplus \sO_{\PP^1}^{\oplus k-1})$ and a $\PP^k$; or
\item a scroll $\PP(\sO_{\PP^1}(a) \oplus \sO_{\PP^1}(1)^{\oplus 2} \oplus \sO_{\PP^1}^{\oplus k-2})$ where $a \in \N_{>0}$ and $\Delta=\PP(\sO_{\PP^1}(1)^{\oplus 2} \oplus \sO_{\PP^1}^{\oplus k-2})$; or
\item a scroll $\PP(\sO_{\PP^1}(a) \oplus \sO_{\PP^1}(2) \oplus \sO_{\PP^1}^{\oplus k-1})$ where $a \in \N_{>0}$
and $\Delta=\PP(\sO_{\PP^1}(2) \oplus \sO_{\PP^1}^{\oplus k-1})$; or
\item a scroll $\PP(V)$ over an elliptic curve and $\Delta = \PP(W)$ where $V \rightarrow W$ 
is a quotient bundle of rank $k$ such that $\det W \simeq \mathcal O_W$.
\end{enumerate}
\end{proposition}

\begin{proof}
By Lemma \ref{lemmacontraction} there exists a $K_X$-negative extremal ray $R$ 
 such that $\Delta \cdot R>0$ and $H \cdot R>0$. Thus if \holom{\varphi}{X}{Y} denotes the contraction of this extremal ray,
the nefvalue $r:= r(\varphi, H)$ (cf. Definition \ref{definitionnefvalue}) is strictly larger than $k$.
Arguing by contradiction 
we suppose that the extremal contraction $\varphi$ is birational. 
Let $F$ be a non-trivial $\varphi$-fibre, then by \cite[Thm.2.1(II,i)]{And95} we have $\dim F \geq \lfloor r \rfloor \geq k= \dim X-1$. Thus $\varphi$ contracts the divisor $F$ onto a point, in particular $F$ meets the Gorenstein locus
of $X$. By Lemma \ref{lemmarounddown} this implies that $r = \lfloor r \rfloor= k$, a contradiction.
Thus $\varphi$ is of fibre type and we can apply \cite[Sect.7.2, 7.3]{BS95} to see that
$X$ is isomorphic to one of the following varieties:
\begin{enumerate}
\item $(\mathbb P^{k+1}, \mathcal O_{\mathbb P^{k+1}}(1))$; or
\item $(Q^{k+1}, \mathcal O_{Q^{k+1}}(1))$; or
\item a generalised cone over $(\mathbb P^2, \mathcal O_{\mathbb P^2}(2))$; or
\item a scroll over a curve $C$.
\end{enumerate}
In the cases a)-c) we are obviously finished, so suppose that $X \cong \PP(V)$
with $V$ a vector bundle of rank $k+1$. Note that $\Delta$ has exactly one irreducible component $\Delta_1$ that surjects onto $C$. Since $\Delta_1 \rightarrow C$ is flat and every fibre is a hyperplane in $\PP^k$, we see that
$\Delta_1 \cong \PP(W)$ with $V \rightarrow W$ a quotient bundle of rank $k$. Set $\Delta':=\Delta-\Delta_1$. 
By the adjunction formula 
$$
-(K_{\Delta_1}+(\Delta_1 \cap \Delta')) \simeq - k H|_{\Delta_1}
$$
and by the canonical bundle formula
$$
K_{\Delta_1} = \varphi|_{\Delta_1}^*(K_C+\det W) - k H|_{\Delta_1}
$$
thus we obtain $-(\Delta_1 \cap \Delta')= \varphi|_{\Delta_1}^*(K_C+\det W)$. Since $\det W$ is nef, this implies
that $K_C$ is antinef. If $C$ is a rational curve we obtain the cases d)-f).
If $C$ is an elliptic curve, the divisors $-(\Delta_1 \cap \Delta')$ and $\det W$ must be trivial.
Thus we obtain the case g).
\end{proof}

\begin{proposition} \label{propositionMRC}
Let $X$ be a Fano variety of dimension $n$ and index $k$ with lc singularities such that $-K_X \simeq kH$, with $H$ an ample Cartier divisor on $X$.
Assume that $\dim \Nklt(X) = k$.
Let \holom{\mu}{X'}{X} be a terminal modification of $X$, and write
$$
K_{X'} \simeq \mu^* K_X - E,
$$ 
where $E$ is an effective, reduced, $\mu$-exceptional Weil divisor.
Suppose that the base of the MRC-fibration $X' \dashrightarrow Z$ has dimension $n-k-1$, and let $F$ be a general fibre. 
Then the log-pair $(F, F \cap E)$ quasi-polarised by the nef and big divisor $(\mu^* H)|_F$ 
is isomorphic to one of the varieties {\rm a)-f)} in Proposition \ref{propositionindexk}.
Moreover this statement is effective, i.e. there exist examples realising all these cases.
\end{proposition}

\begin{proof}
 Since $X$ has lc singularities
and $K_{X'}+E \sim_\Q \mu^* K_X$, the pair $(X', E)$ is lc. By hypothesis $\dim Z=n-k-1$,
so the general fibre $F$ is a $(k+1)$-dimensional variety with  terminal singularities. 
Moreover the pair
$(F, E|_F)$ is lc and $E|_F \neq 0$ since otherwise $X$
is not rationally connected modulo the non-klt locus, in contradiction to \cite[Thm.1.2]{BP11}.
By adjunction we have $K_F+E|_F \sim_\Q -k (\mu^* H)|_F$, moreover any curve $C \subset F$
such that $(\mu^* H)|_F \cdot C=0$ is $\mu$-exceptional, so $K_F \cdot C=K_{X'} \cdot C \geq 0$.
Thus Proposition \ref{propositionindexk} applies.  Case g) is excluded since $F$ is rationally connected.

The statement is effective, the examples corresponding to the varieties a)-f) are:
Examples \ref{exampledimk}, \ref{exampledimkquadric}, \ref{exampleveronese}, \ref{examplereduciblequadric}
and \ref{exampleprojectivebundles} (twice). 
\end{proof}

\begin{remark*}
An analogous statement of Proposition \ref{propositionMRC} should hold
if we replace a terminal modification by the canonical modification. However this would make 
it necessary to prove Proposition \ref{propositionindexk} and thus \cite[Sect.7.2, 7.3]{BS95} for varieties with canonical singularities,
a rather tedious exercise.
\end{remark*}

The following lemma is an analogue of classical descriptions of projective bundles
as in \cite[Prop.3.2.1]{BS95}:

\begin{lemma} \label{lemmaquadricbundle}
Let $Z$ be a projective manifold, and let $X$ be a normal projective variety
admitting an equidimensional fibration \holom{\varphi}{X}{Z} of relative dimension $k$. Assume that  
 there exists an ample Cartier divisor $H$ on $X$ such that the general
polarised fibre $(F, \sO_F(H))$ is isomorphic to the quadric $(Q^k, \sO_{Q^k}(1))$.
Then $X \rightarrow Z$ is a quadric fibration\footnote{We use the definition of an
adjunction theoretic quadric fibration \cite[3.3.1]{BS95} which is a-priori weaker
than supposing that all the fibres are quadrics.}, i.e. there exists a Cartier divisor $M$ on $Z$ such that
$K_X+k H \sim_\Q \varphi^* M$.
\end{lemma}

\begin{proof}
Let $C \subset Z$ be a complete intersection of $\dim Z-1$ general hyperplane sections. The preimage
$X_C := \fibre{\varphi}{C}$ is a normal projective variety and the fibration $\varphi_C \colon X_C \rightarrow C$ 
satisfies the conditions of \cite[Lemma 2.6]{a23}. Thus we know that $X_C$
has  canonical singularities, and there exists a Cartier divisor $M_C$ on $C$ such that 
$$
K_{X_C}+k H|_{X_C} \sim_\Q \varphi_C^* M_C.
$$
Using the canonical modification of $X$ we see that there exists 
a closed (maybe empty) subset $Z' \subset Z$ of codimension at least two in $Z$ such that $X_0:= \fibre{\varphi}{Z \setminus Z'}$ has 
canonical singularities, and $(K_{X}- k H)|_{X_0}$ is nef with respect to the equidimensional fibration 
$\varphi_0\colon X_0 \rightarrow Z_0:=Z \setminus Z'$. It follows from Zariski's lemma \cite[Lemma 8.2]{BHPV04} that
$$
(K_{X}+k H)|_{X_0} \sim_\Q \varphi_0^* M_0
$$
where $M_0$ is a Cartier divisor on $Z_0$. Since $Z$ is smooth the divisor $M_0$ extends to a Cartier divisor $M$ on $Z$. Since $X \setminus X_0$ has codimension at least two in $X$, the isomorphism above extends to
$K_X+k H \sim_\Q \varphi^* M$.
\end{proof}

\begin{proof}[Proof of Proposition \ref{propositionfibrespace}]
By hypothesis  the
base of the MRC-fibration $X' \dashrightarrow Z$ has dimension $n-k-1$. Thus Proposition \ref{propositionMRC}
applies, and the general fibre $F$ is given by the cases a)-f) of Proposition \ref{propositionindexk}. 

If we are in case a), let $\Hilb$ be a desingularisation of the unique component of the cycle space $\Chow{X}$
such that the general point
parametrises $\mu(F)$, and let $\holom{q}{\Gamma}{\Hilb}$ be the normalisation of the pull-back of the universal family.
By construction the natural morphism
$p: \Gamma \rightarrow X$ is birational and $p^* H$ is $q$-ample and its restriction to the general fibre is the hyperplane divisor.
By \cite[Prop.4.10]{AD12}, \cite[Prop.3.5]{a18} the fibration $q$ is a projective bundle. If we are in the cases d)-f) 
the MRC-fibration factors generically through an almost holomorphic fibration $X' \dashrightarrow W$ such that
the general fibre is a linear $\PP^{k-1}$, so we use the same argument to construct a $\PP^{k-1}$-bundle
$\Gamma \rightarrow \Hilb$ that dominates $X$. 

If we are in case b) let $\Hilb$ be a desingularisation of the unique component of the cycle space $\Chow{X}$ 
such that the general point
parametrises $\mu(F)$, and let $\holom{q}{\Gamma}{\Hilb}$ be the normalisation of the pull-back of the universal family.
Then $q$ is a quadric fibration by Lemma \ref{lemmaquadricbundle}. 

If we are in case c) and $k=1$ the general fibre of the MRC-fibration is the Veronese surface. We repeat the
construction above to obtain $\Hilb$ and $\Gamma \rightarrow \Hilb$. If $k \geq 2$ the $(k+1)$-dimensional cone
over the Veronese surface is dominated by the projective bundle $\PP(\sO_{\PP^2}(2) \oplus \sO_{\PP^2}^{\oplus k-1})$.
Thus we can repeat the construction of case a) to obtain a $\PP^{k-1}$-bundle $\Gamma \rightarrow \Hilb$ that dominates $X$.
\end{proof}

\section{The base of the MRC-fibration}

Proposition \ref{propositionfibrespace} gives a rather precise description of the Fano variety $X$, but the condition
on the MRC-fibration seems rather restrictive. In this section we will give strong evidence that this
condition is always satisfied when $\Nklt(X)$ is a quadric, in particular we prove
Theorem \ref{theoremkone}.

The following lemma generalises a part of \cite[Thm.2]{Ume81} to arbitrary dimension:

\begin{lemma} \label{lemmanotvanishing}
Let $Y$ be a projective scheme of pure dimension $m \geq 2$ 
such that the dualising sheaf is trivial. Suppose that there exists an irreducible component $Y_1 \subset Y$
having two points $\{p_1, p_2 \}$ which are not contained in any other component of
$Y$ such that $Y_1$ is normal near $\{p_1, p_2 \}$,  and that  $\{p_1, p_2 \}$ are isolated non-klt points. 
Let $\holom{\mu}{Y'}{Y}$ be the birational morphism defined by the canonical modification of $Y_1$
in the points $p_1$ and $p_2$.
If $Y$ is log-canonical in $p_1$ and $p_2$, then
$h^{m-1}(Y', \sO_{Y'}) \geq 1$. 
\end{lemma}

\begin{proof}
The image of the $\mu$-exceptional locus equals $\{p_1, p_2 \}$,
so the higher direct image sheaves $R^j \mu_* \sO_{Y'}$ have support in a finite set.
Our goal is to prove that we have
\begin{equation} \label{notvanishing}
h^0(Y, R^{m-1} \mu_* \sO_{Y'}) \geq 2.
\end{equation}
Assuming this for the time being, let us see how to conclude by a Leray spectral sequence computation: 
since the sheaves $R^j \mu_* \sO_{Y'}$
have no higher cohomology for $j>0$ we see that
$$
H^0(Y, R^{m-1} \mu_* \sO_{Y'}) = E_2^{0,m-1}=E_{m}^{0, m-1}
$$
and
$$
\C \cong H^0(Y, \omega_Y) \cong H^m(Y, \sO_Y) = E_2^{m, 0} = E_m^{m, 0}.
$$
By \eqref{notvanishing} the first space has dimension at least two. Thus the kernel of the map
$$
d_m : E_{m}^{0, m-1} \rightarrow E_m^{m, 0}
$$
has dimension at least one, hence $\dim E_{m+1}^{0, m-1} = \dim E_{\infty}^{0, m-1}>0$.
Since the map $H^{m-1}(Y', \sO_{Y'}) \rightarrow E_{\infty}^{0, m-1}$ is surjective the statement follows.

For the proof of \eqref{notvanishing} note first that the claim is local in a neighbourhood of
the points $\{p_1, p_2 \}$. Thus we can suppose without loss of generality that $Y$ is a normal variety
with trivial canonical divisor such that $\Nklt(Y)=\{p_1, p_2 \}$ and  $\mu\colon Y' \rightarrow Y$ is the canonical 
modification. Since $Y'$ has canonical, hence rational singularities, we can replace it with a desingularisation,
which for simplicity's sake we denote by the same letter. Since $K_Y$ is Cartier we can write
$$
K_{Y'} \simeq \mu^* K_Y - E + F,
$$
where $E$ is a reduced divisor mapping surjectively onto $\Nklt(Y)$ and $F$ is an effective divisor
such that $E$ and $F$ have no common components. In particular $\omega_E \simeq \sO_E(F)$
is effective. By Kov\'acs' vanishing theorem \cite[Cor.6.6]{Kov11} we have
$R^j \mu_* \sO_{Y'}(-E)=0$ for all $j>0$, hence
$$
R^j \mu_* \sO_{Y'} \simeq R^j (\mu|_E)_* \sO_{E}
$$
for all $j>0$. By duality (in the sense of  \cite[III, Sect.7, Defn.,p.241]{Har77}) we have
$$
H^{m-1}(E, \sO_E) \cong H^0(E, \omega_E).
$$
We have seen that $\omega_E$ is effective. Since $E$ has at least two connected components the
inequality \eqref{notvanishing} follows.
\end{proof}

\begin{definition} \label{definitionladder}
Let $X$ be a Fano variety of dimension $n$ and index $k$ with lc singularities, and let $H$ be a Cartier divisor such that 
$-K_X \simeq kH$. We say that $X$ admits a {\rm ladder} if there exist $k$ general divisors $D_1, \ldots, D_k$ in $|H|$ such that
for all $i \in \{1, \ldots, k\}$ the intersection
$$
Z_i := D_1 \cap \ldots \cap D_i
$$
is a normal variety of dimension $n-k$ with lc singularities such that
$$
\Nklt(Z_i) = \Nklt(X) \cap D_1 \cap \ldots \cap D_i.
$$
\end{definition}

\begin{proposition} \label{propositionnotRC}
Let $X$ be a Fano variety of dimension $n$ and index $k$ with lc singularities, 
and let $H$ be a Cartier divisor such that $-K_X \simeq kH$. 
Suppose that $(\Nklt(X), \sO_{\Nklt(X)}(H))$ is a quadric of dimension $k$. If $k>1$ suppose also that $X$ admits a ladder.
Then we have
$$
h^{n-k-1}(X', \sO_{X'}) \neq 0,
$$
where $X' \rightarrow X$ is the canonical modification. Moreover the base of the MRC-fibration
$X' \dashrightarrow Z$ has dimension $n-k-1>0$.
\end{proposition}

\begin{proof}
Note first that by the Nadel vanishing theorem the restriction map
\begin{equation} \label{surjective}
H^0(X, \sO_X(H)) \rightarrow H^0(\Nklt(X), \sO_{\Nklt(X)}(H))
\end{equation}
is surjective. Since  $\sO_{\Nklt(X)}(H)$ is globally generated, this implies that $|H|$ is globally generated near $\Nklt(X)$.
In particular if $Y \in |H|$ is a general divisor, then we can apply the adjunction 
formula \cite[Lemma 1.7.6]{BS95}  to see that $\omega_{Y} \simeq \sO_{Y}((k-1)H)$. Moreover the intersection
$$
Y \cap \Nklt(X)
$$
is a quadric of dimension $k-1$
and $Y \cap \Nklt(X)$ is a connected component of the non-klt locus of $Y$\footnote{The divisor $Y$ may have
several irreducible components, but by Bertini's theorem there exists a unique irreducible component that intersects
$\Nklt(X)$. Note that this component is normal near $\Nklt(X)$ so it makes sense to speak of the non-klt locus.}. 
Let $\holom{\mu}{X'}{X}$ be the canonical modification, then 
the strict transform $Y'$ of $Y$ coincides with the total transform, hence we have $Y' \simeq \mu^* H$. 
Moreover the birational map $\holom{\mu|_{Y'}}{Y'}{Y}$ is the canonical modification of $Y$
along $Y \cap \Nklt(X)$.

Since $\mu^* H$ is nef and big, we have 
$$
H^j(X', \sO_{X'}(-\mu^* H))=0
$$
for all $j \leq n-1$ by \cite[Thm.2.70]{KM98}. Since $n-k-1 \leq n-2$ we obtain
$$
H^{n-k-1}(X', \sO_{X'}) \twoheadrightarrow H^{n-k-1}(Y', \sO_{Y'}).
$$
We will now conclude by induction: if $k=1$, then $\omega_{Y}$ is trivial and 
$\holom{\mu|_{Y'}}{Y'}{Y}$ satisfies the conditions of Lemma \ref{lemmanotvanishing}.
Thus we have $H^{n-2}(Y', \sO_{Y'}) \neq 0$.
If $k \geq 2$, then by hypothesis the Fano variety $X$ admits a ladder $D_1=Y, D_2, \ldots, D_k$, so $Y$ has lc singularities,
$\Nklt(Y)$ is a quadric and the divisors 
$$
D_i \cap Y \in |H|_{Y}|
$$
define a ladder on $Y$. Thus the induction hypothesis applies to $Y$.

This also implies that the base of the MRC-fibration has dimension at least $n-k-1$.
In order to see that equality holds let us first deal with the case $k=1$. As above let $Y \in |H|$ be a general divisor,
and let $Y_1 \subset Y$ be the unique irreducible component that meets $\Nklt(X)$. Let $Y_1'$ (resp. $Y'$) be
the strict transform under the canonical modification $\mu$. Since $\Nklt(X) \cap Y= \Nklt(X) \cap Y_1$ is not empty
and $\omega_Y \simeq \sO_Y$ we see that $\omega_{Y'}$ is antieffective. Moreover we have a map
$\omega_{Y_1'} \rightarrow \omega_{Y'} \otimes \sO_{Y_1}$ that is an isomorphism in the generic point
of $Y_1'$, so $\omega_{Y_1'}$ is antieffective.
Thus we see that $Y_1'$ is uniruled. Since $Y_1'$ is general (note that by \eqref{surjective} we have $h^0(X, \sO_X(H)) \geq 3$),
it is not contracted by the MRC-fibration $X' \dashrightarrow Z$. 
Since $Z$ is not uniruled \cite{GHS03}, we obtain $\dim Z<\dim Y_1'=n-1$.

If $k>1$ then by hypothesis we can consider  the complete intersection
$Z_{n-k}$ defined in Definition \ref{definitionladder}. Then $K_{Z_{n-k}}$ is trivial
and by what precedes its strict transform $Z_{n-k}' \subset X'$ is uniruled. 
As in the case $k=1$ we obtain $\dim Z<\dim Z_{n-k}=n-k$.
\end{proof}

Theorem \ref{theoremkone} is now an immediate consequence:

\begin{proof}[Proof of Theorem \ref{theoremkone}]
The dimension of the base of the MRC-fibration is a birational invariant for varieties with canonical singularities 
\cite{HMK07}, so by Proposition \ref{propositionnotRC} we have $\dim Z=n-2$. Thus we can conclude
with Proposition \ref{propositionMRC}.
\end{proof}

For Fano varieties with high index we can verify the ladder condition in Proposition \ref{propositionnotRC}:

\begin{proposition} \label{propositionladder}
Let $X$ be a Fano variety of dimension $n$ and index $k \geq n-3$ with  lc singularities,
and let $H$ be a Cartier divisor such that  $-K_X \simeq kH$. Suppose that $\dim \Nklt(X)=k$.
Let $Y \in |H|$ be a general divisor. Then $Y$ is
a normal variety with lc singularities such that
$$
\Nklt(Y) = \Nklt(X) \cap Y.
$$
\end{proposition}

\begin{proof}
By the Nadel vanishing theorem the restriction map
$$
H^0(X, \sO_X(H)) \rightarrow H^0(\Nklt(X), \sO_{\Nklt(X)}(H))
$$
is surjective. By Theorem \ref{theoremnkltlocus} we know that $\sO_{\Nklt(X)}(H)$ is globally generated, so
we have
\begin{equation} \label{intersectionempty}
\Bs |H| \cap \Nklt(X) = \emptyset.
\end{equation}
We claim that the pair $(X \setminus \Nklt(X), Y \setminus \Nklt(X))$ is plt. 
Assuming this for the time being, let us see how to conclude. 
Since the pair $(X \setminus \Nklt(X), Y \setminus \Nklt(X))$ is plt we know by 
inversion of adjunction \cite[Thm.7.5.1]{Kol95} that $Y \setminus \Nklt(X)$ has  canonical singularities.
Since $H|_{X \setminus \Bs |H|}$ is a free linear system, a general $Y \in |H|$ is normal by \cite[Thm.1.7.1]{BS95}
and $(X \setminus \Bs |H|, Y \setminus \Bs |H|)$ is an lc pair.
Using the adjunction formula
\cite[Lemma 1.7.6]{BS95} we see that for a general $Y \in |H|$ we have
$-K_Y \simeq (k-1) H|_Y$.
By inversion of adjunction $Y$ has  lc singularities \cite[Thm. p.130]{Kaw07} and
 $$
\Nklt(Y) = \Nklt(X) \cap Y. 
 $$

{\em Proof of the claim.} We will deal with the case $k=n-3$, the other cases being simpler. We follow the argument of \cite[Thm.1.1]{Flo13}: note that $X \setminus \Nklt(X)$ has  canonical singularities.
We argue by contradiction and suppose that the pair $(X \setminus \Nklt(X), Y \setminus \Nklt(X))$ is not plt.
Then there exists a $0 < c \leq 1$
and an irreducible variety $W \subset \Bs |H|$ such that the pair $(X \setminus \Nklt(X), c(Y \setminus \Nklt(X)))$ 
is properly lc and $W$ a minimal lc centre. By \cite[Thm.2.2]{Fuj11} 
the restriction map
$$
H^0(X, \sO_X(H)) \rightarrow H^0(W, \sO_W(H))
$$
is surjective. 
Since $W$ is contained in the base locus this implies that $H^0(W, \sO_W(H))=0$.
By Kawamata subadjunction \eqref{subadjunction} 
there exists an effective divisor $\Delta_W$ such that $(W, \Delta_W)$ is klt and 
$$
(K_W+\Delta_W) \sim_\Q (K_X+c Y)|_W \sim_\Q -(n-3-c) H|_W.
$$
If $\dim W \leq 2$ we apply \cite[Prop.4.1]{Kaw00} to see that $h^0(W, \sO_W(H)) \neq 0$, a contradiction.
If $\dim W \geq 3$ and $n-3-c>\dim W-3$,
then $(W,B_W)$ is log Fano of index at least $n-3-c$.
By \cite[Theorem 5.1]{Kaw00} this implies $h^0(W,\mathcal{O}_W (H))\neq 0$, a contradiction.
Thus we are left with the case when $\dim W \geq 3$ and 
$\dim W\geq n-c$. Since $c \leq 1$ this implies $\dim W = n-1$. 
Since the centre $W$ is minimal we know by \cite[Lemma 2.7]{Flo13}
that $c<1/2$. Thus we have $\dim W \geq n-1/2$, a contradiction.
\end{proof}

\section{Examples} \label{sectionexamples}

In \cite[$\S 2$]{BS87} the first-named author and Sommese observed that for a Fano threefold
of index one with $\dim \Nklt(X)=1$, the non-klt locus consists of rational curves. They also ask
how far $X$ can deviate from being a generalised cone with $\Nklt(X)$ as its vertex. In this section
we answer this question: all the classification results of Section \ref{sectiondescriptionX} are effective,
so in many cases $X$ is far from being a generalised cone.

\begin{example} \label{exampledimkminusone}
Let $Y$ be a projective manifold with trivial canonical divisor, and let $L$ be a very ample line bundle on $Y$. Set
$$
X' := \PP( \sO_Y^{\oplus k} \oplus L ),
$$
then we have
$$
K_{X'} = -k \zeta - (\zeta - \varphi^* L),
$$
where $\zeta$ is the tautological divisor on $X'$. We have $\zeta - \varphi^* L=E$ where $E \cong Y \times \PP^{k-1}$ is the 
divisor defined by the quotient
$$
\sO_Y^{\oplus k} \oplus L \rightarrow \sO_Y^{\oplus k}. 
$$
The divisor $\zeta$ is globally generated and defines a birational morphism $\holom{\mu}{X'}{X}$ that contracts $E$ onto 
a $\PP^{k-1}$. Using the canonical bundle formula above we see that $X$ is a Fano variety of index $k$ with lc singularities,
moreover the non-klt locus is a linear $\PP^{k-1}$.
\end{example}

\begin{example}  \label{exampledimk} 
Let $Y$ be a projective manifold admitting a fibration $\holom{\psi}{Y}{\PP^1}$
such that $\sO_Y(-K_Y) \simeq \psi^* \sO_{\PP^1}(1)$. Then set
$X':= \PP(A \oplus \psi^* \sO_{\PP^1}(1) \oplus \sO_Y^{\oplus k-1})$ where $A$ is a very ample line bundle on $Y$.
Let $\zeta$ be the tautological divisor on $X'$, then $\zeta$ is globally generated and defines
a birational morphism \holom{\mu}{X'}{X} contracting the divisor $E$ corresponding to the quotient
$$
A \oplus \psi^* \sO_{\PP^1}(1) \oplus \sO_Y^{\oplus k-1} \rightarrow \psi^* \sO_{\PP^1}(1) \oplus \sO_Y^{\oplus k-1}
\simeq \psi^* (\sO_{\PP^1}(1) \oplus \sO_{\PP^1}^{\oplus k-1})
$$
onto a $\PP^k$. Using the canonical bundle formula one easily sees that $X$ is Fano of index
$k$ with lc singularities and $(\Nklt(X), \sO_{\Nklt(X)}(H)) \cong (\PP^k, \sO_{\PP^k}(1))$,  where $-K_X\simeq kH$.
\end{example}

\begin{example} \label{exampledimkquadric}
 Let $X$ be a   Fano variety of index $m$ constructed as in  Example \ref{exampledimkminusone}, and let $-K_X\simeq mH$
 with $H$ a Cartier divisor on $X$.
Set $k=m-1$.
By the Nadel vanishing
theorem the restriction map
$$
H^0(X, \sO_X(2H)) \rightarrow H^0(\Nklt(X), \sO_{\Nklt(X)}(2H))
$$
is surjective. Fix now a quadric $Q \subset \PP^k \cong \Nklt(X)$ and fix a general divisor
$B$ in the linear system $|2 H|$ such that $B \cap \Nklt(X) = Q$.
Denote by
$$
\holom{\pi}{\widetilde X}{X}
$$
the cyclic covering of degree two branched along $B$\footnote{Note that by
the generality assumption the divisor $B$ is reduced even if $Q$ is not reduced. In particular
$\widetilde X$ is irreducible.}. By the ramification formula we see that
$-K_{\widetilde X} = \pi^* k H$,
so $\widetilde X$ is a Fano variety of index $k$. 
Using again the ramification formula we know by \cite[Prop.5.20]{KM98} that $\widetilde X$ has lc singularities
if and only if the pair $(X, \frac{1}{2} B)$ is lc. Since $B$ is general this is clear in the complement
of $\Nklt(X)$, moreover by inversion of adjunction \cite[Thm.0.1]{Hac12} the pair  $(X, \frac{1}{2} B)$
is lc near $\Nklt(X)$ if and only if $(\PP^k, \frac{1}{2} Q)$ is lc. Since $Q$ is a quadric this is easily checked.
The restriction of $\pi$ to $\Nklt(\widetilde{X}) \cong \PP^k$ induces a two-to-one cover
$$
\Nklt(\widetilde X) \rightarrow \Nklt(X)
$$
that ramifies exactly along $Q = B \cap \Nklt(X)$. Since $Q$ is a quadric in $\PP^k$, we see that 
$\Nklt(\widetilde X)$ is a quadric, the singularities depending on the singularities of $Q$.
Let $X' \rightarrow Y$ be the projective bundle dominating $X$ (cf. Example \ref{exampledimkminusone}).
Then $X' \times_X \widetilde X \rightarrow Y$ is a quadric bundle dominating $\widetilde X$.
\end{example}

\begin{example} \label{exampleveronese}
Let $Y$ be a projective manifold with trivial canonical divisor. 
Let $L$ be a very ample line bundle on $Y$ and set
$$
B' := \PP(L \oplus \sO_Y^{\oplus 2}).
$$
Denote by $\eta$ the tautological divisor on $B'$, then
$2 \eta$ is globally generated and defines a birational morphism $\holom{\nu}{B'}{B}$
onto a normal projective variety contracting the divisor $D$ corresponding to the quotient
$$
L \oplus \sO_Y^{\oplus 2} \rightarrow \sO_Y^{\oplus 2}
$$
onto a $\PP^1$. Using the canonical bundle formula for $B'$ we see that
$$
K_{B'} \simeq -2 \eta -D \simeq \nu^* K_B - D,
$$
so $B$ has lc singularities and $(\Nklt(B), \sO_{\Nklt(B)}(-K_B)) \cong (\PP^1, \sO_{\PP^1}(2))$.
For some integer $k \geq 2$ we set 
$$
X' := \PP(\sO_{B'}(2 \eta) \oplus \sO_{B'}^{\oplus k-1}), 
$$
and denote by $\zeta$ the tautological divisor on $\holom{\varphi}{X'}{B'}$. The divisor $\zeta$ is globally generated and defines a birational morphism $\holom{\mu}{X'}{X}$. 
If we restrict $\mu$ to  $\fibre{\varphi}{\PP^2}$ where $\PP^2$ is a fibre of $B' \rightarrow Y$,
the image is a generalised cone of dimension $k+1$ over the Veronese surface $(\PP^2, \sO_{\PP^2}(2))$.
If we restrict $\mu$ furthermore to $\fibre{\varphi}{\PP^2 \cap D}$, the image is the generalised cone
of dimension $k$ over the line $\PP^2 \cap D$, but polarised by $\sO_{\PP^2}(2)$, so we
get a quadric $Q$ of dimension $k$ such that the singular locus has dimension $k-2$.
Since
$D \simeq \PP(\sO_Y^{\oplus 2})$,
a straightforward computation shows that 
$$
\zeta^{k+1} \cdot \varphi^* D = 0 \ \mbox{in} \ H^{2k+4}(X', \R),
$$
so the divisor
$E:=\varphi^* D$ is contracted by $\mu$ onto the $k$-dimensional quadric $Q$. Using the canonical bundle
formula we see that
$$
K_{X'} = -k \zeta - E = \mu^* K_X - E,
$$
so $X$ is a Fano variety of index $k$ having lc singularities and $\Nklt(X) \cong Q$. 
Note also that $\mu$
factors through a birational morphism $\holom{p}{\Gamma}{X}$ where $\Gamma$ is a normal
projective variety admitting a locally trivial fibration $\holom{q}{\Gamma}{Z}$
such that the polarised fibre is the generalised cone of dimension $k+1$ over the Veronese surface $(\PP^2, \sO_{\PP^2}(2))$.
\end{example}

\begin{example} \label{examplereduciblequadric}
For some $k \geq 3$, let $Y \subset \PP^{k-1}$ be a smooth hypersurface of degree $k$,
so $Y$ is a Calabi--Yau manifold.  We set
$$
B' := \PP(\sO_Y(1) \oplus \sO_Y),
$$
and denote by $\eta$ the tautological divisor on $B'$. Then $\eta$ is globally generated
and defines a birational morphism \holom{\nu}{B'}{B} where $B$ is the cone over $Y$.
We have $K_{B'} \simeq \nu^* K_{B} - D$, where $D$ is the divisor corresponding to the quotient
$$
\sO_Y(1) \oplus \sO_Y \rightarrow \sO_Y.
$$
Note that $B$ is a hypersurface of degree $k$ in $\PP^{k}$, so $B$ is Cohen--Macaulay
and has exactly one non-klt point, the vertex of the cone. The anticanonical sheaf 
$\sO_B(-K_{B}) \simeq \sO_{B}(1)$ is globally generated and $h^0(B, \sO_B(-K_{B}))=k+1$.
We define $W^*$ to be the vector bundle of rank $k$ that is the kernel of
the evaluation  map
$\sO_{B}^{\oplus k+1} \twoheadrightarrow \sO_B(-K_{B})$. Dualizing we obtain an exact sequence
$$
0 \rightarrow \sO_B(K_{B}) \rightarrow \sO_{B}^{\oplus k+1} \rightarrow W \rightarrow 0,
$$
so $W$ is globally generated. Moreover by a singular version of Kodaira vanishing \cite[Cor.6.6]{KSS10}
we have $H^1(B, \sO_B(K_{B}))=0$, so we get $h^0(B, W)=k+1$.

Set now $W':=\nu^* W$ and let $A$ be the pull-back of a very ample Cartier divisor on $B$.
We set
$$
X' := \PP(\sO_{B'}(A) \oplus W'),
$$
and denote by $\zeta$ the tautological divisor on $\holom{\varphi}{X'}{B'}$. The divisor $\zeta$ is globally generated and defines a birational morphism $\holom{\mu}{X'}{X}$. 
Denote by $E \subset X'$ the divisor defined by the quotient
$$
\sO_{B'}(A) \oplus W' \rightarrow W'.
$$
Then by the canonical bundle formula
$$
K_{X'} \simeq \varphi^*(K_{B'}+A+\det W')-(k+1) \zeta = -k \zeta - (\zeta-\varphi^* A) + \varphi^* (K_{B'}+\det W').
$$
By construction we have $\zeta-\varphi^* A \simeq E$ and $K_{B'}+\det W' \simeq - D$, so we get
$$
K_{X'} \simeq -k \zeta - E - \varphi^* D.
$$
Since $D$ is contracted by $\nu$, the restriction of $\sO_{B'}(A) \oplus W'$ to $D$ is isomorphic
to $\sO_D^{\oplus k+1}$, so $\varphi^* D$ is contracted by $\mu$ onto a $\PP^k$.
The divisor $E$ is isomorphic to $\PP(W')$ and by construction $h^0(B', W')=k+1$,
so $E$ is also contracted onto a $\PP^k$. Thus $X$ is a Fano variety of index $k$
with lc singularities such that $\Nklt(X)$ is a reducible quadric.
\end{example}

\begin{example} \label{exampleprojectivebundles}
Let $Y$ be a projective manifold with trivial canonical divisor and set $B := Y \times \PP^1$.
Let $A$ be a very ample Cartier divisor on $B$ and set
$$
X' := \PP(\sO_B(A) \oplus \sO_B(-K_B) \oplus \sO_B^{\oplus k-1}).
$$
Denote by
$\zeta$ the tautological divisor on $\holom{\varphi}{X'}{B}$. The divisor $\zeta$ is globally generated and defines a birational morphism $\holom{\mu}{X'}{X}$. 
Moreover if $E \subset X'$ is the divisor defined by the quotient
$$
\sO_B(A) \oplus \sO_B(-K_B) \oplus \sO_B^{\oplus k-1} \rightarrow \sO_B(-K_B) \oplus \sO_B^{\oplus k-1},
$$
then $E$ is contracted by $\mu$ onto a quadric $Q$ of dimension $k$ that is singular along a subvariety
of dimension $k-2$. By the canonical bundle formula we see that
$$
K_{X'} = -k \zeta - E = \mu^* K_X - E,
$$
so $X$ is a Fano variety of index $k$ having lc singularities and $\Nklt(X) \cong Q$. 

Analogously, if we set
$$
X' := \PP\big(\sO_B(A) \oplus \sO_B(- \frac{1}{2}K_B)^{\oplus 2} \oplus \sO_B^{\oplus k-2}\big),
$$
the same properties hold; in this case the quadric is singular along a subvariety of dimension $k-3$.
\end{example}

\def\cprime{$'$}

\end{document}